\documentclass[12pt]{amsart}
\usepackage{graphicx, color}
\usepackage{amsmath,amscd, amsfonts}
\usepackage{graphicx, color}
\usepackage{amssymb, amsmath,amsthm}
\usepackage{epstopdf}
\newtheorem{thm}{{\bf Theorem}}

\newtheorem{cor}[thm]{{\bf Corollary}}
\newtheorem{prop}[thm]{{\bf Proposition}}

\numberwithin{equation}{section}

\newenvironment{psmallmatrix}
  {\left(\begin{smallmatrix}}
  {\end{smallmatrix}\right)}

\def\CC{{\mathfrak C}}
\def\FF{{\mathfrak F}}  
\def\F{{\mathbb F}}

\def\Z{{\mathbb Z}}

\def\a{{\alpha}}
\def\b{{\beta}}
\def\c{{\gamma}}

\begin{document} 

\title{Ordered groups as a tensor category}

\date{\today}

\author{Dale Rolfsen}
\address{Department of Mathematics\\
The University of British Columbia\\
Vancouver, BC, Canada V6T 1Z2}
\email{rolfsen@math.ubc.ca}

\begin{abstract} 
It is a classical theorem that the free product of ordered groups is orderable.  In this note we show that, using a method of G. Bergman, an ordering of the free product can be constructed in a functorial manner, in the category of ordered groups and order-preserving homomorphisms.   With this
functor interpreted as a tensor product this category becomes a tensor (or monoidal) category.   Moreover, if $O(G)$ denotes the space of orderings of the group $G$ with the natural topology, then for fixed groups $F$ and $G$ our construction can be considered a function $O(F) \times O(G) \to O(F * G)$.  We show that this function is continuous and injective.  Similar results hold for left-ordered groups.
\end{abstract}

\thanks{The author gratefully acknowledges the support of a grant from the Canadian Natural Sciences and Engineering Research Council.  I also thank George Bergman, Adam Clay and Christian Kassel for very helpful comments on earlier versions of this paper.  Thanks also to Victoria Lebed and Arnaud Mortier for providing me with an English translation of  \cite{Vinogradov49}}

\maketitle

\section{Introduction}

An ordered group $(G, <)$ is a group $G$ together with a strict total ordering $<$ of its elements such that $x < y$ implies $xz < yz$ and $zx < zy$ for all $x, y, z \in G$.  If such an ordering exists, $G$ is said to be orderable.  If $(F, <_F)$ and $(G, <_G)$ are ordered groups, a homomorphism
$\phi: F \to G$ is said to be order-preserving (relative to $<_F, <_G$) if for all $x, y \in F$, 
$x <_F y$ implies $\phi(x) <_G \phi(y)$.  Note that the reverse implication follows, and that such a $\phi$ is necessarily injective.

A theorem of Vinogradov \cite{Vinogradov49} asserts that if $F$ and $G$ are orderable groups, then the free product $F*G$ (sometimes called the coproduct, as in \cite{Bergman90})  is orderable.  Other proofs of this can be found in \cite{Johnson68}, \cite{Passman85} and \cite{Bergman90}, and a generalization in \cite{Chiswell12}. 
A proof given in \cite{MR77} was unfortunately found to have a gap, as discussed in \cite{HM94} and \cite{Chiswell14}.  Yet another proof, in \cite{Revesz87}, was also shown to have a gap \cite{Medvedev91}.

In this note, we show that a version of Bergman's construction in \cite{Bergman90} is functorial in the following sense.  Suppose $(F_i, <_{F_i}), i= 0, 1,$ are ordered groups.  We will construct an ordering $\prec$ of $F_0*F_1$, so that $(F_0*F_1, \prec)$ is an ordered group, and write  
$${\mathfrak F }((F_0, <_{F_0}) , (F_1, <_{F_1})) := (F_0*F_1, \prec).$$
  Theorem \ref{main} shows that $\mathfrak F$ is a (bi-)functor in the category $\mathfrak{C}$ of ordered groups and order-preserving homomorphisms.  We will show in Section \ref{tensor} that this functor gives 
  $\CC$ the structure of a  tensor, or monoidal, category. 
  
\begin{thm}\label{main}
Suppose that $(F_i, <_{F_i}), i = 0, 1,$  are ordered groups.  Then the ordered group 
$(F_0*F_1, \prec_F) = {\mathfrak F }((F_0, <_{F_0}), (F_1, <_{F_1}))$ has the following properties:

(1) $\prec_F$ extends the given orderings of $F_i$ as subgroups of $F_0 * F_1$ and 

(2) if $(G_i, <_{G_i}), i = 0, 1,$ are ordered groups and $(G_0 *G_1, \prec_G) = {\mathfrak F }((G_0, <_{G_0}), (G_1, <_{G_1}))$ and if
$\phi_i: F_i \to G_i, i = 0, 1,$ are homomorphisms which preserve the given orderings of $F_i$ and $G_i$, then 
the homomorphism $\phi_0* \phi_1 : F_0*F_1 \to G_0*G_1$ is order-preserving, relative to $\prec_F, \prec_G$.
\end{thm}

In Section \ref{arbitrary index}, Theorem \ref{main} will be extended to free products of an arbitrary, possibly infinite, collection of ordered groups.  We will typically use multiplicative notation for groups and use $1$ to denote the identity element, though additive groups are also considered, with 0 as identity element.  We may also use $1$ to denote the unit of a ring (all rings we consider are assumed to have a unit), as well as the natural number.

Many of our results could have been proven using the original construction of Vinogradov.  Like Bergman's, his proof involves embedding a free product of groups into a ring of matrices.  Vinogradov's matrices are infinite dimensional upper triangular matrices, whereas Bergman's are 2 by 2 matrices with polynomial entries, a useful simplification.

%
%

\section{Embedding free products in matrix rings}\label{construction}
We use an observation of Bergman which generalizes the fact that the matrices 
$\begin{psmallmatrix} 1 & t \\ 0 & 1 \end{psmallmatrix}$ and $\begin{psmallmatrix} 1 & 0 \\ t & 1 \end{psmallmatrix}$ freely
generate a free subgroup of the multiplicative group of invertible $2 \times 2$ matrices with entries in the polynomial ring $\Z[t]$.

Consider a ring $R$ without zero divisors and let $F$ and $G$ be multiplicative groups of nonzero elements of $R$.  Let $M_2(R[t])$ be the ring of $2 \times 2$ matrices with entries in the polynomial ring $R[t]$.  Then one can embed $F$ in $M_2(R[t])$ by $f \mapsto \begin{psmallmatrix} f & 0 \\ 0 & 1 \end{psmallmatrix}$.  But we can conjugate that by $\begin{psmallmatrix} 1 & t \\ 0 & 1 \end{psmallmatrix}$ to get a different embedding which has a highest degree in the upper right corner when $f \ne 1$:
$$ \rho(f) =
 \begin{pmatrix} 1 & -t \\ 0 & 1 \end{pmatrix} \begin{pmatrix} f & 0 \\ 0 & 1 \end{pmatrix}
\begin{pmatrix} 1 & t \\ 0 & 1 \end{pmatrix} = \begin{pmatrix} f & (f - 1)t \\ 0 & 1 \end{pmatrix}.
$$
Similarly we embed $G$ by $$\rho(g) = \begin{pmatrix} 1 & 0 \\ (g-1)t & g \end{pmatrix}.$$  This then defines a 
multiplicative homomorphism $\rho : F * G \to M_2(R[t])$, which Bergman observes to be a faithful representation.

\begin{prop}[\cite{Bergman90}, Corollary 12]\label{injective}
With the assumptions stated in the preceding paragraph, $\rho : F*G \to M_2(R[t])$ is injective.
\end{prop}

\begin{proof}
Here is a sketch of a proof using a ping-pong argument.  Let $f_kg_kf_{k-1} \cdots g_2f_1g_1 \ne 1$ be a reduced word in $F*G$, with $f_i \in F, g_i \in G$ nonidentity elements (except possibly the first and/or last).  Assume that $g_1 \ne 1$, the other case with $g_1 = 1, f_1 \ne 1$ being similar.
We need to show that the product of matrices $\rho(f_k)\rho(g_k) \cdots \rho(f_1)\rho(g_1)$ is not the identity matrix.  Consider the set $V$  of column vectors $\begin{psmallmatrix} A(t) \\ B(t) \end{psmallmatrix}$ with entries in $R[t]$ and partition that set into three parts $V = V_1 \sqcup V_2 \sqcup V_3$ according to their degrees as polynomials.  Take $V_1$ to be the set of such pairs with $deg A(t) > deg B(t), V_2$ the set with $deg A(t) < deg B(t)$ and  $V_3$ the set with equal degree. 

Apply $\rho(f_k)\rho(g_k) \cdots \rho(f_1)\rho(g_1)$ (on the left) to the vector $\begin{psmallmatrix} 1 \\ 1 \end{psmallmatrix} \in V_3$ and note that $\rho(g_1)$ sends $\begin{psmallmatrix} 1 \\ 1 \end{psmallmatrix}$ to 
$\begin{psmallmatrix} 1 \\ g_1 + (g_1 - 1)t \end{psmallmatrix}$ which belongs to $V_2$.  Then $\rho(f_1)$ sends this result into $V_1$, which is then sent to $V_2$ by $\rho(g_2)$, and so on.  The end result, after multiplying all the matrices, will be in $V_1$ or $V_2$, not $V_3$, and so the product cannot be the identity matrix. 
\end{proof}

\section{Constructing the ordering $\prec$}\label{constructing}

Suppose we are given two ordered groups, 
$(F_0, <_{F_0})$ and $(F_1, <_{F_1})$.  To embed them in a ring, we take $R$ to be the integral group ring of their direct product:
$R = \Z(F_0 \times F_1)$.  It is well-known that integral group rings of orderable groups have no zero divisors (see, for example, \cite{MR77} p. 155), so $R$ has no zero divisors..
Define a multiplicative homomorphism $\rho: F_0 * F_1 \to M_2(R[t])$ by

$$ \rho(f_0) =    \left(\begin{array}{cc}f_0 & (f_0 - 1)t \\0 & 1\end{array}\right) \quad
\rho(f_1) =  \left(\begin{array}{cc}1 & 0 \\ (f_1 - 1)t & f_1\end{array}\right), \quad f_i \in F_i.
$$

By proposition \ref{injective}, $\rho$ is faithful; it defines an isomorphism of $F_0 * F_1$ onto a multiplicative subgroup of $M_2(R[t])$.

We now turn to the task of defining the ordering, choosing a specific recipe among many described in \cite{Bergman90}.  First we order 
$F_0 \times F_1$ lexicographically, defining
$( f_0, f_1) < (f_0' , f_1')$ if $f_0 <_{F_0} f_0'$ or else $f_0 = f_0'$ and $f_1 <_{F_1} f_1'$.  Then the group ring $R = \Z(F_0 \times F_1)$ becomes an ordered ring\footnote{We understand an ordered ring $(R, <)$ to be an ordered group as an additive group, for which the positive cone $P = \{ r \in R \mid 0 < r\}$ is also closed under multiplication.}
by declaring a nonzero element
to be {\em positive} if the coefficient of the largest term (in the ordering  $<$ of $F_0 \times F_1$) is a positive integer.

Note that as a ring element, $f_0 \in F_0$, which can be considered an abbreviation of $1(f_0, 1) \in R,$ is considered positive even if $f_0 <_{F_0} 1$ and it  would be called ``negative'' as a group element.  In particular, the diagonal elements of the matrices displayed above are all positive.

Bergman then orders $M_2(R[t])$ as follows.  Choose ``an arbitrary order among the four `positions' in a $2 \times 2$ matrix, and call a nonzero element of this module `positive' if in the first position in which a nonzero coefficient occurs, the coefficient is in fact positive.''  He points out that ``The orderings of the positions can be the same for all $n$, but need not -- there is a lot of freedom here.''  To be definite, we will choose for all $n$ the $1,1$ position to be first, the $2,2$ position to be second, and the off-diagonal positions ordered third and fourth in some fixed way.  

Call an element $M$ of $M_2(R[t])$ positive if satisfies the following.  Expand $M = M_0 + M_1t + \cdots + M_kt^k$, where each $M_i$ belongs to $M_2(R)$.  Let $n \ge 0$ be the least integer such that $t^n$ has nonzero coefficient and
say $M$ is positive iff the first nonzero entry of $M_n$ is positive in the ordered ring $R$.

Finally, define an ordering of $F_0 * F_1$ by declaring that $x \prec y$ if and only if $\rho(y) - \rho(x)$ is positive in 
$M_2(R[t])$.

\section{Proof of Theorem \ref{main} and further properties of $\prec$}
First we'll argue that $(F_0*F_1, \prec)$ is an ordered group.  Clearly $\prec$ is a strict total ordering.  To check invariance under multiplication, first note that every element of $\rho(F_0*F_1)$ in $M_2(R[t])$, when expanded in powers of $t$, has constant term a diagonal matrix with positive entries.  (See the proof of Proposition \ref{order-homomorphism} below to be more precise.) The product of such a matrix, on either side, with a positive matrix in $M_2(R[t])$ will again be positive.  Thus, if $x, y, z \in F_0*F_1$, one has $x \prec y \iff
\rho(y) - \rho(x)$ is positive $ \iff  \rho(z)(\rho(y) - \rho(x)) = \rho(zy) - \rho(zx)$ is positive $ \iff zx \prec zy$.  Right invariance is proved similarly. Next we will show that the ordering $\prec$ extends the given orderings $<_{F_0}$ and $<_{F_1}$.
Suppose $f_0, f_0' \in F_0$ and $f _0 <_{F_0} f_0'$.  Then their images in $M_2(R[t])$ have
difference the matrix
$ \begin{psmallmatrix}f_0' - f_0 & * \\ 0 & 0\end{psmallmatrix}$, and noting that $f_0' - f_0$ is positive in $R$ we conclude $f_0 \prec f_0'$.  A similar argument shows that $\prec$ also extends $<_{F_1}$.

This establishes the first part of Theorem \ref{main}.  To prove part (2), note that $\phi_0 \times \phi_1$ preserves the lexicographic orderings $<_F, <_G$ of $F_0 \times F_1$ and $G_0 \times G_1$, respectively.  A homomorphism of groups naturally extends to a ring homomorphism of the integral group rings, and we see that if the group homomorphism preserves given orderings of the groups, then its extension takes ``positive'' elements of the group ring to positive elements.  Then $\phi_0 \times \phi_1$ defines a ring homomorphism $R_F \to R_G$, where $R_F = \Z(F_0\times F_1)$ and $R_G = \Z(G_0 \times G_1)$, which we will call $\phi_0 \times \phi_1$ again.  This extends to a ring homomorphism $R_F[t] \to R_G[t]$, and further induces an additive homomorphism $M_2(R_F[t]) \to M_2(R_G[t])$, which we will again call $\phi_0 \times \phi_1$.

The diagram 
$$\begin{CD}
F_0*F_1                                        @>{\rho}>>                                 M_2(R_F[t]) \\
@V{\phi_0* \phi_1}VV                                         @V{\phi_0 \times \phi_1}VV \\
G_0*G_1                                       @>{\rho}>>                               M_2(R_G[t]) 
\end{CD}$$
is commutative (we have used the same symbol $\rho$ for different maps, but defined analogously), and as already mentioned, $\phi_0 \times \phi_1$ takes positive matrix entries to positive matrix entries.  We now argue that $\phi_0 * \phi_1$ is order-preserving, relative to $\prec_F, \prec_G$.  Suppose $x, y \in F_0 * F_1$ 
and $x \prec_F y$. Then $\rho(y) - \rho(x)$ is positive, and therefore 
$\phi_0 \times \phi_1(\rho(y) - \rho(x))$ is positive in $M_2(R_G[t])$.  But 
$\phi_0 \times \phi_1(\rho(y) - \rho(x)) = \phi_0 \times \phi_1(\rho(y)) - \phi_0 \times \phi_1(\rho(x)) =
\rho(\phi_0 * \phi_1(y)) - \rho(\phi_0 * \phi_1(x)),$ and since this is positive, we conclude that 
$\phi_0 * \phi_1(x) \prec_G \phi_0 * \phi_1(y).$ \qed

\begin{cor}\label{automorphisms}
If $(F, <_F)$ and $(G, <_G)$ are ordered groups, then the ordered group $(F*G, \prec) := {\mathfrak F}((F, <_F), (G, <_G))$  has the properties that 
$\prec$ extends the orderings of $F$ and $G$, and for any automorphisms $\phi:F \to F$ and $\psi:G \to G$ which preserve the given orderings, the automorphism $\phi * \psi : F*G \to F*G$ preserves the ordering $\prec$.
\end{cor}

Following the terminology used in \cite{MR77}, we will call a homomorphism $\phi: F \to G$ of ordered groups $(F, <_F)$ and $(G, <_G)$ an {\em order-homomorphism} (relative to the given orderings) if  $x \le_F y$ implies $\phi(x) \le_G \phi(y)$ for all $x, y \in F$.  
Note that order-preserving homomorphisms are order-homomorphisms, and that order-homomorphisms need not be injective.  Indeed, the order-preserving homomorphisms are exactly the order-homomorphisms which are injective.
For example, using the lexicographic ordering of the direct product, the inclusions $F \to F \times G$ and $G \to F \times G$ are order-preserving, while the projection
$F \times G \to F$ is an order-homomorphism. But the projection $F \times G \to G$ will not be an order-homomorphism, if the groups are nontrivial.  

We'll see that our construction of $\prec$ has similar properties.  First note that part (1) of Theorem \ref{main} implies that the natural inclusion homomorphisms $F \to F * G$ and $G \to F * G$ are order-preserving.  There are also canonical maps $F * G \to F$, obtained by killing elements of $G$, and similarly $F * G \to G$.  They combine to define a canonical homomorphism $\alpha: F * G \to F \times G$.  Specifically, if 
$f_1g_1f_2 \cdots f_kg_k$ is an element of $F * G$, with $f_i \in F$ and $g_i \in G,$ then 
$\alpha(f_1g_1f_2 \cdots f_kg_k) = (f_1 \cdots f_k, g_1 \cdots g_k).$

\begin{prop}\label{order-homomorphism}
Suppose that $(F, <_F)$ and $(G, <_G)$ are ordered groups.  Then the canonical homomorphism 
$\alpha \colon F * G \to F \times G$ 
is an order-homomorphism, relative to the lexicographic ordering of $F \times G$ and the ordering $\prec$ for $F * G$.  
\end{prop}

\begin{proof} 
If $x \in F*G$ has image $\alpha(x) = (f, g) \in F\times G$, we observe that its image under the representation  
$\rho : F*G \to M_2(R[t])$ may be written $\rho(x) = \begin{psmallmatrix}f & 0 \\ 0 & g\end{psmallmatrix} +$ terms of positive degree.  The conclusion follows from our convention for ordering $M_2(R[t])$.
\end{proof}

A subset $C \subset G$ of an ordered group $(G, <_G)$ is said to be {\em convex} if the inequalities $c <_G g <_G c'$, with $c, c' \in C$ imply that 
$g \in C$.  For example, it is easy to see that if $(F, <_F)$ and $(G, <_G)$ are ordered groups and $\phi: F \to G$ is an order-homomorphism, then the kernel $K$ of $\phi$ is a convex subgroup of $F$.

\begin{cor}\label{convex kernel} 
The kernel of the homomorphism $\alpha \colon F * G \to F \times G$ is convex, relative to the ordering $\prec$ of $F * G$.
\end{cor}

The kernel of $\alpha \colon F * G \to F \times G$ is known to be a free subgroup of $F * G$, freely generated by commutators of the form $fgf^{-1}g^{-1},$ where $1 \ne f \in F$ and  $1 \ne g \in G.$
  
\begin{cor} \label{asymmetry} If $F * G$ is ordered by $\prec$, the canonical homomorphism $F * G \to F$ is an order-homomorphism, but $F * G \to G$ will not be an order-homomorphism, if the groups are nontrivial.
\end{cor}

Indeed, if $f <_F f'$ in $F$ while $g' <_G g$ in $G$, we have, as elements of $F * G$ the inequality $fg \prec f'g'$.  If the canonical map $F*G \to G$ were an order-homomorphism, we'd conclude $g <_G g'$, a contradiction.  The asymmetry exposed by this corollary cannot be corrected, as the following observation shows.  We will not need it, and leave the proof to the interested reader.

\begin{prop}\label{asymmetry_necessary}
If $F$ and $G$ are nontrivial ordered groups, then there is no ordering of $F*G$ for which both of the  canonical homomorphisms  $F * G \to F$ and  $F * G \to G$ are order-homomorphisms.  
\end{prop}

\section{Structure as a tensor category}\label{tensor}

Recall that $\CC$ denotes the category of ordered groups and order-preserving homomorphisms, and that $\FF : \CC \times \CC \to \CC$ is a bi-functor.  Let us rename $\FF$ as follows, for ordered groups 
$(F_0, <_{F_0})$ and $(F_1, <_{F_1})$:
$$(F_0, <_{F_0}) \otimes (F_1, <_{F_1}) := \FF((F_0, <_{F_0}), (F_1, <_{F_1})) = (F_0 * F_1, \prec)$$

It is well-known that the category of groups under free product is a tensor category, with unit the trivial group (see, for example, \cite{Mac98} or the Wikipedia entry for Monoidal Category).  I am grateful to Christian Kassel for suggesting the following to me.

\begin{thm}\label{thm-tensor}
With the bi-functor $\otimes$ the category $\CC$ is a tensor category, in other words a monoidal category.
\end{thm}

For ordered groups $(F_0, <_{F_0}), (F_1, <_{F_1}), (F_2, <_{F_2})$, we have the isomorphism of groups
$$F_0 * (F_1 * F_2) \cong (F_0 * F_1) * F_2.$$ 
We need to check that the orderings constructed on both sides of this equivalence are the same under the isomorphism, in other words the isomorphism is order-preserving. But this follows from the observation that the lexicographic orderings on the direct products 
$F_0 \times (F_1 \times F_2)$ and $(F_0 \times F_1) \times F_2$, used in the respective orderings of 
$F_0 * (F_1 * F_2)$ and $(F_0 * F_1) * F_2,$ both reduce to the lexicographic ordering of triples.

Similarly, the coherence relations involved in tensor categories follow from the observation that for ordered groups $(F_i, <_{F_i}), 0 \le i \le 3,$ our orderings of the groups
$$(F_0 * F_1) * (F_2 * F_3),\; (F_0 * (F_1 * F_2)) * F_3,\;  F_0 * ((F_1 * F_2) * F_3),$$
$$(F_0 * F_1) * (F_2 * F_3), {\rm and \;} F_0 * (F_1 * (F_2 * F_3))  $$
are identical (under their natural isomorphisms).

%


\section{An application to braid groups}

The original motivation for this study is the following application to the theory of braids.
The braid group $B_n$ acts by automorphisms on the free group $\F_n$, as observed by Artin \cite{Artin25, Artin47}.  Free groups are orderable, and we may call a braid ``order-preserving'' if its image under the (faithful) Artin representation $B_n \to Aut(\F_n)$ preserves {\em some} ordering of $\F_n$ (see  \cite{KR16}).  In that paper it is noted that a braid is order-preserving if and only if the complement of the link in $S^3$ consisting of the braid's closure, plus the braid axis, has orderable fundamental group.  It is used to show, for example, that of the two minimal volume orientable hyperbolic 2-cusped 3-manifolds, one has orderable fundamental group, while the group of the other is not orderable (although it is left-orderable). 

\begin{center}
\begin{figure}[h]
\includegraphics[width=3in]{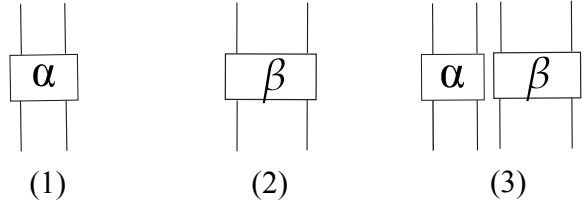}
\caption{(1) $\alpha \in B_m$. 
(2)  $\beta \in B_n$. 
(3) $\alpha \otimes \beta \in B_{m+n}$.}
\label{fig_tensorProduct}
\end{figure}
\end{center}

Multiplication of braids is by concatenation, and the product of two order-preserving braids need not be order-preserving, as observed in  \cite{KR16}.   There is also a tensor product operation $\otimes: B_m \times B_n \to B_{m+n}$ which forms an $m+n$ strand braid $\a \otimes \b$ from an $m$-braid $\a$ and an $n$-braid $\b$ by placing them side by side with no crossing between the strands of $\a$ and those of $\b$, as in Figure \ref{fig_tensorProduct}.  See for example \cite{KT08}, p. 69.

It is easy to see from the definition of Artin's representation that the automorphism of $\F_{m+n} \cong \F_m * \F_n$ corresponding to $\a \otimes \b$ is just the free product of the automorphisms corresponding to $\a$ and $\b$.  

\begin{cor}
The tensor product $\a \otimes \b$  of braids is order-preserving if and only if both $\a$ and $\b$ are order-preserving braids.
\end{cor} 

\begin{proof}
One direction follows from Corollary \ref{automorphisms}.  For if $\a$ and $\b$ preserve some orderings of $\F_m$ and $\F_n$ respectively, then $\a \otimes \b$ preserves the corresponding ordering $\prec$ of $\F_m * \F_n \cong \F_{m+n}$.
On the other hand, suppose $\a \otimes \b$ preserves an ordering of $\F_{m+n} \cong \F_m * \F_n $  Considering $\F_m$ and $\F_n$ as the natural subgroups of $\F_m * \F_n $, we see that the action of  $\a \otimes \b$  leaves each of these subgroups invariant.  Therefore the ordering of $\F_{m+n}$ preserved by $\a \otimes \b$ restricts to each of the subgroups making the action of the braids $\a$ and $\b$ order-preserving. 
\end{proof}

We note the multiple use of the tensor product symbol.  Indeed, let us say that the ordered free group 
$(\F_n, <)$ {\em represents} the braid $\b \in  B_n$ if the automorphism of $\F_n$ corresponding to $\b$ under the Artin representation preserves the ordering $<$.  We have observed the following.

\begin{prop}
If $(\F_m, <)$ represents $\a \in B_m$ and $(\F_n, <')$ represents $\b \in B_n$, then 
$(\F_m, <) \otimes (\F_n, <')$ represents $\a \otimes \b \in B_{m+n}$.
\end{prop}

\section{Continuity}\label{continuity}

The goal of this section is to establish that our construction is continuous in an appropriate sense.  
If $O(G)$ denotes the set of all (two-sided invariant) orderings of the group $G$, there is a natural topology on $O(G)$, defined below.  Given orderable groups $F$ and $G$, the construction defined in Section \ref{constructing} can be considered a function whose input is a pair of orderings $<_F$ and $<_G$ and the output is an ordering $\prec$ of $F * G$, in other words a function 
$O(F) \times O(G) \to O(F * G)$.  We'll see that it is both continuous and injective.
 
\subsection{The space of orderings}

The set of orderings $O(G)$ of the group $G$ is endowed with a natural topology, as detailed by Sikora \cite{Sikora04}.  Consider a specific ordering $<_G$ of $G$, and choose a {\em finite} number of inequalities among elements of $G$ which are satisfied using $<_G$.
Then a basic neighbourhood of $<_G$ consists of all orderings of $G$ for which all those inequalities remain true.  Neighbourhoods of this type form a basis for the topology we are considering.  Equivalently, a neighbourhood of $<_G$ is defined by choosing some finite set of elements of $G$ which are positive (greater than the identity) using $<_G$.  Then take the neighbourhood to consist of all orderings of $G$ under which that finite 
set remains positive.

It is known, and not difficult to show, that $O(G)$ is compact and totally disconnected.  An isolated point of $O(G)$ is an ordering which is ``finitely determined'' in the sense that it is the only ordering of $G$ for which some finite set of inequalities holds.  Sikora \cite{Sikora04} showed that for $n \ge 2, O(\Z^n)$ has no isolated points, and is homeomorphic with the Cantor set.  Whether $O(\F_n)$ has isolated points, for the free group 
$\F_n, n \ge 2$, is an open question at this writing.  

\subsection{Continuity of lexicographic ordering of direct products.}

As a warmup to our main result, we consider the lexicographic ordering of direct products $F \times G$ of ordered groups, as discussed in Section \ref{constructing}  (similar results would hold for the reverse lex ordering).  It may be considered a function
$$\mathfrak{L} \colon O(F) \times O(G) \to O(F \times G).$$

\begin{prop}\label{lex continuous} 
$\mathfrak{L}$ is continuous and injective.
\end{prop}

\begin{proof} We may assume both $F$ and $G$ are nontrivial groups; otherwise there is nothing to prove.
For injectivity, suppose $<_F$ and $<_F'$ are orderings of $F$ and that $<_G$ and $<_G'$ are orderings of $G$.  Consider 
$< \; = \mathfrak{L}(<_F, <_G)$ and $<' \; = \mathfrak{L}(<_F', <_G')$.  If $<_F$ and $<_F'$ are distinct, there must be an element
$f \in F$ with $1 <_F f$ but $f <_F' 1$.  Then we have, for any $g \in G$, that $1 < (f, g)$ and $(f, g) <' 1$.  It follows that $<$ and $<'$ are distinct.  Similarly, if  $<_G$ and $<_G'$ are different, then one can find an element $(1, g) \in F * G$ with $(1, g)$ having different signs relative to the orderings $<$ and $<'$.  This establishes injectivity.
  
To establish continuity, note that
a basic neighbourhood $\mathfrak{N}_<$ of $<$ in $O(F \times G)$ is defined by choosing some finite set of positive elements: 
$$(f_1, g_1), \dots, (f_k, g_k), (1, g_{k+1}), \dots (1, g_{k+l}).$$
Here we have 
$$1<_F f_1, \dots,  1 <_F f_k \quad {\rm and} \quad1 <_G g_{k+1}, \dots , 1 <_G g_{k+l},$$ 
whereas some of the list $g_1, \dots, g_k$ may be negative in the ordering $<_G$.  Possibly $k = 0$ or $l = 0$.

Continuity will be established if we can find neighbourhoods $\mathfrak{N}_{<_F}$ of $<_F$ in $O(F)$ and $\mathfrak{N}_{<_G}$ of $<_G$ in $O(G)$ so that 
$\mathfrak{L}(\mathfrak{N}_{<_F} \times \mathfrak{N}_{<_G}) \subset \mathfrak{N}_<$.  But this is straightforward: take 
$\mathfrak{N}_{<_F}$ to be the set of all orderings of $F$ for which $f_1, \dots f_k$ are positive, and 
$\mathfrak{N}_{<_G}$ the set of all orderings of $G$ under which $g_{k+1}, \dots, g_{k+l}$ are positive.
\end{proof}

\subsection{Continuity of the ordering of free products}
Recalling the construction in Section \ref{constructing}, we defined a function of ordered groups:
$${\mathfrak F }((F, <_F) , (G, <_G)) = (F * G, \prec).$$

By abuse of notation, if $F$ and $G$ are fixed, but orderings thereof are variable, we may write 
$$\mathfrak{F} (<_F, <_G) = \; \prec .$$
Then we have a function of spaces of orderings:
$$\mathfrak{F} \colon O(F) \times O(G) \to O(F * G)$$

\begin{thm}\label{cont and inj}
$\mathfrak{F}$ is continuous and injective.
\end{thm}

\begin{proof} One may prove injectivity as in Proposition \ref{lex continuous}; we leave the details to the reader.  Note also that we proved continuity of the map $\mathfrak{L}$ by showing that any finite set of inequalities in $F \times G$ would be implied (under 
$\mathfrak{L}$) by finitely many inequalities in $F$ and in $G$.

We will argue similarly in this case; we'll try to avoid excessive notation and sketch the ideas.  Suppose $<_F$ and $<_G$ are given orderings of $F$ and $G$, respectively, and that $\prec \; = \mathfrak{F}(<_F, <_G)$ is the corresponding ordering of the free product $F * G$.  A neighbourhood $\mathfrak{N}_{\prec}$ of $\prec$ in the space $O(F * G)$ consists of all orderings of $F * G$ for which all members of some finite set $x_1, \dots, x_k$ of elements of $F * G$ are positive, where $1 \prec x_i$ for $i = 1, \dots, k$.  But note that $1 \prec x_i$ is equivalent to the matrix $\rho(x_i) - \rho(1)$ being positive in $M_2(\Z(F \times G)[t])$, and this is positive if the first nonzero entry of that matrix, expanded in powers of $t$, is positive.  That entry, an element of $\Z(F \times G)$, is positive if the coefficient of its greatest group element, say $(f_i, g_i)$, is a positive integer.  
But the condition that $(f_i, g_i)$ is the greatest group element appearing in that entry is equivalent to a finite number of inequalities in $F \times G$, using the lexicographic ordering.  This in turn, as in Proposition  \ref{lex continuous}, is implied by a finite number of inequalities in $F$ and $G$ which are in particular satisfied using the orderings $<_F$ and $<_G$.  Using the open neighbourhoods $\mathfrak{N}_{<_F}$ of $<_F$ and $\mathfrak{N}_{<_G}$ of $<_G$ defined by those inequalities, we see that 
$\mathfrak{F}(\mathfrak{N}_{<_F}, \mathfrak{N}_{<_G}) \subset \mathfrak{N}_{\prec}$, which establishes continuity of $\mathfrak{F}.$
\end{proof}

Suppose, in the procedure for defining $\prec$ in Section \ref{constructing}, one used some ordering of $F \times G$ other than the lexicographic one, but otherwise defined $\prec$ in the same way.  This then defines a function $O(F \times G) \to O(F * G)$, which we will call $\mathfrak{M}$, short for matrix construction.  The proof of Theorem \ref{cont and inj} actually shows that $\mathfrak{M}$ is continuous.  Our specific construction 
$\mathfrak{F}$ may therefore be considered  a composite 
$$O(F) \times O(G) \xrightarrow{\mathfrak{L}} O(F \times G) \xrightarrow{\mathfrak{M}} O(F * G)$$
of two continuous functions, both injective.

\section{Free product of arbitrarily many ordered groups}\label{arbitrary index}

We now consider an arbitrary collection of ordered groups.  For convenience, we assume the groups are indexed by an ordinal number $\c$ and denote the collection by $\{(F_\a, <_{F_\a})\}_{\a < \c}.$  So far we have been considering the case $\c = 2$.

\begin{thm}\label{many groups}
Let $\c \ge 2$ be an ordinal.  Suppose $\{(F_\a, <_{F_\a})\}_{\a < \c}$ is a collection of ordered groups and let $F := *_{\a<\c} F_\a$ denote the free product.  Then there is an ordering $\prec_F$ of $F$, so that $(F, \prec_F)$ is an ordered group, denoted
$\mathfrak{F}(\{(F_\a, <_{F_\a})\}_{\a < \c}) :=  (F, \prec_F)$, and such that the following hold:

(1) For each $\a < \c$ the restriction of $\prec_F$ to the natural subgroup $F_\a$ of $F$ equals $<_{F_\a}$.

(2) If $\{(G_\a, <_{G_\a})\}_{\a < \c}$ is another collection of ordered groups with $G := *_{\a < \c}G_\a$ and 
$$(G, \prec_G) = \mathfrak{F}(\{(G_\a, <_{G_\a})\}_{\a < \c}),$$ then for any collection $\phi_\a : F_\a \to G_\a$ of homomorphisms defined for all $\a<\c$ and which are order-preserving, relative to $<_{F_\a}$ and $<_{G_\a}$, the free product homomorphism 
$*_{\a<\c}\phi_\a :  F  \to G$ is order-preserving, relative to $\prec_F$ and $\prec_G.$

\end{thm}

\begin{proof} We will define the ordering of $F$ by induction, possibly transfinite.  For that reason, we'll call the ordering $\prec_\c$ and only later call it $\prec_F$ also.  The base for the induction, for $\c = 2$, is Theorem \ref{main}, taking $\prec_2$ to be the ordering $\prec$ defined there.
For induction we may assume that orderings $\prec_\b$ have been defined for all the groups $*_{\a<\b} F_\a$ for all $1 < \b < \c$, and that they satisfy (1) and (2) with $\b$ replacing $\c$.   Note that $*_{\a<\b} F_\a$ is naturally a subgroup of $*_{\a<\c} F_\a$.  To facilitate the induction, we'll prove that in addition to properties (1) and (2) of the theorem, $\prec_\c$ further satisfies:

(3) Whenever $1 < \b < \c$ the restriction of the ordering $\prec_\c$ to  $*_{\a<\b} F_\a$ coincides with $\prec_\b$.

Again, by Theorem \ref{main} this is satisfied for the base case $\c = 2$.  To construct $\prec_\c$ we consider two cases.  
Case 1: $\c$ is a successor ordinal: $\c = \b + 1$.  Since $\prec_\b$ is by hypothesis already defined, and noting that $F$ can be naturally identified with  $(*_{\a<\b} F_\a)*F_\b$, we use the functor $\mathfrak{F}$ defined in the proof of Theorem \ref{main} and take
$$(F, \prec_\c) \cong ((*_{\a<\b} F_\a)*F_\b, \prec_\c) := \mathfrak{F}((*_{\a<\b} F_\a), \prec_\b), (F_\b, <_\b)).$$ 

Case 2: $\c$ is a limit ordinal.  Then the group $*_{\a<\c} F_\a$ is the union of its subgroups $*_{\a<\b} F_\a$ with $\b < \c$.  Thus to compare two group elements $x, y$ in $*_{\a<\c} F_\a$, choose $\b < \c$ for which $x, y \in *_{\a<\b} F_\a$ and define $x \prec_\c y$ iff
$x \prec_\b y$.  By property (3) which may be assumed for ordinals less than $\c$, this does not depend on choice of $\b$.

In either case, it is routine to verify that the ordering $\prec_\c$ (also called $\prec_F$) satisfies the conditions (1), (2) and (3).
\end{proof}

\section{Left-ordered groups}  An ordering $<$ of the elements of a group $G$ is a left-ordering if for all 
$f, g, h \in G$ one has $g < h \implies fg < fh$; in this case we call $(G, <)$ a left-ordered group.  It is much easier than the ordered case to see that the free product of left-ordered groups is left-orderable.  For left-ordered groups 
$(F, <_F)$ and $(G, <_G)$  consider the short exact sequence
$$1 \to K \to F * G \to F \times G \to 1,$$
where $F*G \to F \times G$ is the canonical homomorphism.  The kernel $K$ is a free group, which is orderable, and one can left-order $F \times G$, lexicographically.  Since left-orderability (unlike orderability) is always preserved under extensions, we conclude that $F * G$ is left-orderable.

On the other hand, our construction of the ordering $\prec$ for the free product of ordered groups may be revised in a straightforward way to the left-ordered (or right-ordered) situation.  One must be a bit careful.  For a left-ordered group $(G, <)$ the group ring $\Z(G)$ is not, strictly speaking, an ordered ring by our definition.  For example if we have $g, g', h \in G$ with $g < g'$ but $gh > g'h$ then the ring elements
$g' - g$ and $h$ are positive, whereas their product $g'h - gh$ is not positive.  However the product in the other order, $hg' - hg,$ is necessarily positive, and more generally a positive element of $\Z(G)$ multiplied on the left by a monomial with positive coefficient remains positive.  This is enough to establish left-invariance of $\prec$ in the proof of Theorem \ref{main}.

Therefore, we conclude that all the results above remain true if ``ordered'' is replaced by 
``left-ordered" throughout.  In particular, the category of left-ordered groups and order-preserving homomorphisms is also a tensor category using our functorial construction.

\section{Concluding remarks}

The ordering we construct is by no means canonical; for example other choices of ordering the direct product, or the entries of matrices, can lead to a different ordering of the free product which satisfies the conditions of  Theorem \ref{main}, and even defines a tensor category structure.  Indeed, Corollary \ref{asymmetry} reveals the asymmetry of the construction.  In a real sense, the first group in the free product of two groups is treated preferentially in our construction.  It could as well have been the reverse.

The argument given here does not extend to the larger category of ordered groups and order-homomorphisms (which are not necessarily injective) as some positive matrix entries may be mapped to zero under such a map.  Extending our results to this category seems to be an open question.

As noted in \cite{Bergman90}, much of this can be done in the more general setting of ordered semigroups; see also \cite{Johnson68}.  We leave such generalization for the interested reader to contemplate.

\bibliographystyle{amsplain}
\bibliography{CoproductOrdering}

\def\cprime{$'$}
\providecommand{\bysame}{\leavevmode\hbox to3em{\hrulefill}\thinspace}
\providecommand{\MR}{\relax\ifhmode\unskip\space\fi MR }
\providecommand{\MRhref}[2]{%
  \href{http://www.ams.org/mathscinet-getitem?mr=#1}{#2}
}
\providecommand{\href}[2]{#2}
\begin{thebibliography}{10}

\bibitem{Artin25}
E.~Artin, \emph{Theorie der {Z\"opfe}}, Abh. Math. Sem. Univ. Hamberg
  \textbf{4} (1925), 47--72.

\bibitem{Artin47}
\bysame, \emph{Theory of braids}, Ann. of Math. (2) \textbf{48} (1947),
  101--126. \MR{0019087 (8,367a)}

\bibitem{Bergman90}
George~M. Bergman, \emph{Ordering coproducts of groups and semigroups}, J.
  Algebra \textbf{133} (1990), no.~2, 313--339. \MR{1067409 (91j:06035)}

\bibitem{MR77}
Roberta Botto~Mura and Akbar Rhemtulla, \emph{Orderable groups}, Marcel Dekker
  Inc., New York, 1977, Lecture Notes in Pure and Applied Mathematics, Vol. 27.
  \MR{0491396 (58 \#10652)}

\bibitem{Chiswell12}
I.~M. Chiswell, \emph{Ordering graph products of groups}, Internat. J. Algebra
  Comput. \textbf{22} (2012), no.~4, 1250037, 14. \MR{2946302}

\bibitem{Chiswell14}
\bysame, \emph{Ordering free products of groups}, Mathematica Slovaca
  \textbf{64} (2014), no.~3, 707--726.

\bibitem{HM94}
W.~C. Holland and N.~Ya. Medvedev, \emph{A very large class of small varieties
  of lattice-ordered groups}, Comm. Algebra \textbf{22} (1994), 551--578.

\bibitem{Johnson68}
R.~E. Johnson, \emph{Free products of ordered semigroups}, Proc. Amer. Math.
  Soc. \textbf{19} (1968), 697--700. \MR{0227279}

\bibitem{KT08}
Christian Kassel and Vladimir Turaev, \emph{Braid groups}, Graduate Texts in
  Mathematics, vol. 247, Springer, New York, 2008, With the graphical
  assistance of Olivier Dodane. \MR{2435235}

\bibitem{KR16}
Eiko Kin and Dale Rolfsen, \emph{Braids, orderings and minimal volume cusped
  hyperbolic 3-manifolds},  (2016), preprint available at arXiv:1619.03241.

\bibitem{Mac98}
Saunders Mac~Lane, \emph{Categories for the working mathematician}, second ed.,
  Graduate Texts in Mathematics, vol.~5, Springer-Verlag, New York, 1998.
  \MR{1712872}

\bibitem{Medvedev91}
N.~Ya. Medvedev, \emph{A remark on a paper by {G}. {R}\'ev\'esz}, Czechoslovak
  Math. J. \textbf{41(116)} (1991), no.~1, 51. \MR{1087621}

\bibitem{Passman85}
Donald~S. Passman, \emph{The algebraic structure of group rings}, Robert E.
  Krieger Publishing Co. Inc., Melbourne, FL, 1985, Reprint of the 1977
  original. \MR{798076 (86j:16001)}

\bibitem{Revesz87}
G{\'a}bor R{\'e}v{\'e}sz, \emph{A simple proof of {V}inogradov's theorem on the
  orderability of the free product of {$o$}-groups}, Czechoslovak Math. J.
  \textbf{37(112)} (1987), no.~2, 310--312. \MR{882601}

\bibitem{Sikora04}
{\relax A.S}.~Sikora, \emph{Topology on the spaces of orderings of groups},
  Bull. London Math. Soc. \textbf{36} (2004), 519--526.

\bibitem{Vinogradov49}
A.~A. Vinogradov, \emph{On the free product of ordered groups}, Mat. Sbornik
  N.S. \textbf{25(67)} (1949), 163--168, English translation by V. Lebed and A
  Mortier available at arXiv:1703.05781. \MR{0031482}

\end{thebibliography}

\end{document}